\documentclass[reqno,english]{amsart}
\usepackage[utf8]{inputenc}
\usepackage{amsfonts, amssymb, amsmath, amsthm}
\usepackage{mathtools}
\usepackage{enumerate}
\usepackage{graphicx}
\usepackage{fullpage}
\usepackage{verbatim}
\usepackage{hyperref}
\hypersetup{colorlinks=true, linkcolor=red, urlcolor=blue, citecolor=blue, pdfpagemode=UseNone, pdfstartview=FitH, bookmarksopen=true}
\hypersetup{pdftitle=Title}

\newtheorem*{theorem*}{Theorem}
\newtheorem{theorem}{Theorem}

\newtheorem*{lemma*}{Lemma}
\newtheorem{lemma}[theorem]{Lemma}
\newtheorem{problem}[theorem]{Problem}
\newtheorem*{proposition*}{Proposition}

\newtheorem*{fact*}{Fact}

\newtheorem*{question*}{Question}

\newtheorem{conjecture}[theorem]{Conjecture}
\newtheorem*{corollary*}{Corollary}
\newtheorem{corollary}[theorem]{Corollary}
\newtheorem*{claim*}{Claim}
\newtheorem{claim}[theorem]{Claim}

\theoremstyle{remark}
\newtheorem*{remark*}{Remark}
\newtheorem{remark}[theorem]{Remark}

\theoremstyle{definition}
\newtheorem*{definition*}{Definition}
\newtheorem{definition}[theorem]{Definition}

\numberwithin{theorem}{section}

\newtheorem*{observation*}{Observation}

\newcommand{\R}{\mathbb{R}}
\newcommand{\Z}{\mathbb{Z}}

\newcommand{\Q}{\mathbb{Q}}
\newcommand{\F}{\mathbb{F}}

\newcommand{\p}{\mathbf{p}}
\newcommand{\bv}{\mathbf{v}}
\newcommand{\fV}{\mathfrak{V}}
\newcommand{\cL}{\mathcal{L}}

\DeclareMathOperator{\SPAN}{span}

\DeclareMathOperator{\rank}{rank}

\newcommand{\unitvec}{{\mathbf 1}}

\newcommand{\lip}{{\llcorner}}

\newif\ifcmnts
\cmntstrue 
\ifcmnts
\newcommand{\marrow}{\marginpar{\boldmath$\longleftarrow$}}
\newcommand{\denys}[1]{\ifhmode\newline\fi\marrow
  \textsf{\textcolor{magenta}{\bf
      Denys:} #1\newline}}
\else
\newcommand{\denys}[1]{}
\fi

\usepackage[colorinlistoftodos,prependcaption,textsize=small]{todonotes}

\begin{document}
\title{Volume rigidity and algebraic shifting}
\author{Denys Bulavka}
\address{Department of Applied Mathematics, Faculty of Mathematics and Physics\\ Charles University\\ Prague\\ Czech Republic.}
\email{dbulavka@kam.mff.cuni.cz}
\thanks{D.B. is supported by the project "Grant Schemes at CU" (reg. no. CZ.02.2.69\slash0.0\slash0.0\slash19\textunderscore073\slash0016935), by the grant no. 21-32817S of the Czech Science Foundation (GA\v{C}R) and by Charles University project PRIMUS\slash21\slash SCI\slash014.}
\author{Eran Nevo}
\address{Einstein Institute of Mathematics\\
 Hebrew University\\ Jerusalem~91904\\ Israel.}
\email{nevo@math.huji.ac.il}
\thanks{E.N. is partially supported by the Israel Science Foundation grant ISF-2480/20.}
\author{Yuval Peled}
\address{Einstein Institute of Mathematics\\
 Hebrew University\\ Jerusalem~91904\\ Israel.}
\email{yuval.peled@mail.huji.ac.il}

\begin{abstract}
We study the generic volume rigidity of $(d-1)$-dimensional simplicial complexes in $\R^{d-1}$, and show that the volume rigidity of a complex can be identified in terms of its exterior shifting. In addition, we establish the volume rigidity of triangulations of several $2$-dimensional surfaces and prove that, in all dimensions 
$>1$, volume rigidity is {\em not} characterized by a corresponding hypergraph sparsity property.
\end{abstract}
\maketitle
\section{Introduction} 
Let $K$ be an $n$-vertex $(d-1)$-dimensional simplicial complex and $\p:V(K)\to\R^{d-1}$ be a generic mapping of its vertices, in the sense that its $(d-1)n$ coordinates are algebraically independent over $\Q$. This paper deals with the infinitesimal version of the following problem: is there a {\em non-trivial} continuous motion of the vertices starting at $\p$ that preserves the volumes of all the $(d-1)$-simplices in $K$? By ``non-trivial" we mean that, for some $(d-1)$-simplex on $V(K)$ that is not in $K$, its volume would change along the motion. 
It is easy to show that the continuous and infinitesimal versions coincide for generic embeddings, as is the case for graph rigidity~\cite{AR1}.

\medskip

\textbf{Volume Rigidity.} 
The signed volume of a $(d-1)$-face $\sigma=\{v_1,\dots ,v_d\} \in K$ with respect to $\p$ is given by the determinant of the $d\times d$ matrix
\begin{equation*}
 M_{\p, \sigma} = 
        \begin{pmatrix}
                 \p(v_1) & \dots & \p(v_d)\\
                1 & \dots &1
        \end{pmatrix}.
\end{equation*}
Observe that , for every $1\le i\le d-1,~1\le j \le d$, the derivative of the signed volume $\det M_{\p,\sigma}$ with respect to the $i$-th coordinate of $\p(v_j)$ is given by the cofactor $C_{i,j} (M_{\p, \sigma})$ --- that is, the determinant of the submatrix obtained by removing the $i$-th row and  $j$-th column multiplied by $(-1)^{i+j}$. 

The {\em volume-rigidity} matrix $\fV(K,\p)$ of the pair $(K,\p)$ is a $(d-1)n\times f_{d-1}(K)$ matrix, where the columns are indexed by the $(d-1)$-faces of $K$, and every vertex is associated with a block of $(d-1)$ rows. The column vector $\bv_\sigma$ corresponding to a $(d-1)$-face $\sigma=\{v_1,\dots ,v_d\} \in K$ is defined by
\[
(\bv_\sigma)_{v_i,j} = C_{i,j}(M_{\p,\sigma})~,~i\in[d],~j\in [d-1],
\]
and $0$ elsewhere. Here $(\bv_\sigma)_{v_i,j}$ denotes the $j$-th coordinate of $\bv_\sigma$ in the block of $v_i$. In words, $\fV$ is the Jacobian of the function $\p\mapsto (\det M_{\p, \sigma})_{\sigma\in K}$, viewing $\p$ as a $(d-1)n$-dimensional vector.

This matrix was introduced in ~\cite[Appendix A]{lubetzky21} along with the description of a {\em trivial} $(d^2-d-1)$-subspace of the left kernel of $\fV(K,\p)$, arising from the volume-preserving transformations of $\R^{d-1}$. Concretely, the trivial subspace consists of all $(d-1)n$-dimensional vectors $z$ obtained by choosing a $(d-1)\times (d-1)$ matrix $A$ whose trace is zero and a vector $u\in\R^{d-1}$, and letting
$
z_v = A\cdot\p(v)+u
$
for every vertex $v$. The following definition suggests itself.
\begin{definition}
An $n$-vertex $(d-1)$-dimensional simplicial complex $K$ is called {\em volume-rigid} if $$\mathrm{rank}(\fV(K,\p)) = (d-1)n - (d^2-d-1),$$ for a generic $\p:V(K)\to\R^{d-1}$.
\end{definition}
\medskip

\textbf{Exterior shifting.} 
Algebraic shifting was introduced by Kalai (see e.g.~\cite{kalai90} and the survey~\cite {Kalai:surveyShifting}) and has been studied extensively in algebraic combinatorics. Here we present a variant of exterior shifting. The standard basis $(e_i)_{i\in [n]}$ of $\R^n$ induces the basis  $(e_{\sigma})_{\sigma \subseteq [n]}$ of its exterior algebra $\bigwedge \R^n$. 
Consider a generic  basis $(f_1,\dots ,f_n)$ of $\R^n$, where without loss of generality we assume that $f_1 = \unitvec \in \R^n$, namely, the other $n^2-n$ coordinates in this basis are algebraically independent over $\Q$. Consider the \emph{exterior face ring} $\bigwedge K = \bigwedge \R^n~/~(e_{\sigma} \colon \sigma\notin K)$, and let $q$ denote the natural quotient map.  Given a partial order $<$ on the power set of $[n]$, define 
\begin{equation}\label{eqn:Delta}
    \Delta^<(K) = \{\sigma \colon q(f_\sigma) \notin \SPAN_{\R} \{q(f_\tau) \colon \tau< \sigma, |\tau|=|\sigma|\}\}.
\end{equation}

Of special importance in our case is the partial order $<_p$  defined by $\sigma=\{\sigma_1<...<\sigma_m\} \le_p \tau=\{\tau_1<...<\tau_{m'}\}$ if 
$m=m'$ and $\sigma_i\le \tau_i,~\forall i\in[m]$. Corollary \ref{cor:shifted} asserts that $\Delta^p(K):= \Delta^{<_p}(K)$ is a shifted simplicial complex independent of the generic choice of $f$. (Note that  $\Delta^p(K)$ may have more faces than  $K$.)

\medskip
\subsection{Main results.}
Our main result is a characterization of volume rigidity in the setting of Kalai's exterior shifting.

\begin{theorem}
\label{thm:main}
Fix $d\ge 3$. An $n$-vertex $(d-1)$-dimensional simplicial complex $K$ is volume-rigid if and only if $\{1,3,4,...,d,n\} \in \Delta^p(K)$.
\end{theorem}

In the $2$-dimensional case we are able to derive the volume rigidity of triangulations of the following surfaces.

\begin{corollary}\label{cor:surfaces}
Every triangulation of the $2$-sphere, the torus, the projective plane or the Klein bottle is volume rigid. 
In addition, every triangulation of the $2$-sphere and the torus minus a single triangle is also volume-rigid. In particular, every simplicial disc with a $3$-vertex boundary is \emph{minimally} volume-rigid. 
\end{corollary}

In the case of the $2$-sphere we give a complete mathematical proof. For the 
other surfaces
, we reduce  --- via edge contractions \'a la Whiteley~\cite{Whiteley-split} --- to irreducible triangulations, whose volume-rigidity we verify numerically. 

Hypergraph sparsity was introduced by Streinu and Theran~\cite{Streinu-Theran09}, generalizing results on graph sparsity, prominently by White and Whiteley~\cite{Whiteley-Some96} who studied it from a matroid perspective. We say that a $(d-1)$-complex is $(d-1,d^2-d-1)$-sparse 
(resp. tight) 
if every subset $A$ of its vertices of cardinality at least $d$ spans at most $(d-1)|A|-(d^2-d-1)$ simplices of dimensions $d-1$ 
(resp. and equality holds when $A$ equals the entire vertex set). 

Clearly, a vertex subset $A$ spanning more $(d-1)$-simplices induces a non trivial linear dependence between the columns of $\fV(K,\p),$ and it is natural to ask whether this characterizes all the linear dependencies in the volume rigidity matrix. Using Theorem \ref{thm:main}, we show that the answer is negative, 
hence a Laman-type condition for volume-rigidity does not hold true~\footnote{Corollary~\ref{cor:Sparsity-vs-VolRigidity} shows that Prop.1 in the preprint~\cite{ST:wrong} from 2007 is a misstatement.}.

\begin{corollary}\label{cor:Sparsity-vs-VolRigidity}
For every $d\ge 3$, there exists a $(d-1,d^2-d-1)$-tight $(d-1)$-complex 
that is not volume-rigid. 
\end{corollary}

\subsection{Relation to previous works.}
The maximal independent sets of columns of $\fV(\binom{[n]}{d},\p)$, for all generic embeddings $\p$, form the bases of the same matroid. For $d=2$ they correspond to spanning trees, namely the bases in the graphic matroid on $\binom{[n]}{2}$. Kalai~\cite{kalai84} introduced for every integer $k\ge 1$ the $k$-hyperconnectivity matroid on $\binom{[n]}{2}$~\footnote{The $k$-hyperconnectivity matroid is derived from an embedding of the vertex set into $\R^k$. Studying higher hyperconnectivity translates to increasing the dimension of the embedding space.}, where $k=1$ corresponds to the graphic matroid, and identified its bases in terms of exterior shifting (w.r.t. the lexicographic order): $G$ is a basis if and only if the edges of $\Delta^{\rm{lex}}(G)$ form the initial segment that ends with $\{k,n\}$, w.r.t. the lex-order.

Here, in Theorem~\ref{thm:main}, rather then increasing the dimension of the embedding space and staying with graphs, we increase also the dimension of the pure complex, by the same number, and characterize the bases of the resulted $d$-volume-rigidity matroid in terms of exterior shifting w.r.t. the partial order $<_p$.  


The fact that $(d-1,d^2-d-1)$-sparse complexes form the independent sets of a matroid on $\binom{[n]}{d}$ was asserted in~\cite{lorea1979matroidal,Whiteley-Some96}. Additional matroidal and algorithmic properties of sparsity matorids were studied by Streinu and Theran in~\cite{Streinu-Theran09,ST10}. By Corollary~\ref{cor:Sparsity-vs-VolRigidity}, the  $(d-1,d^2-d-1)$-sparsity matroid strictly contains the $(d-1)$-volume-rigidity matroid for all $d\ge 3$. It would be interesting to find further combinatorial conditions that once imposed on the bases of the sparsity-matroid would give the bases of the volume-rigidity matroid.

The remainder of the paper is organized as follows. In Section \ref{sec:VRDP} we establish the connection between volume rigidity and exterior shifting, and prove Theorem \ref{thm:main}. Afterwards, in Section \ref{sec:local} we investigate the effect of local moves on volume rigidity and prove Corollary \ref{cor:surfaces}. In the following Section \ref{sec:sparse} we prove Corollary \ref{cor:Sparsity-vs-VolRigidity}, and we conclude in Section \ref{sec:open} with some related open problems.

\section{Volume rigidity and $\Delta^p(\cdot)$}\label{sec:VRDP}
This section is devoted to studying the basic properties of the shifted complex $\Delta^p(K)$, and to establishing the connection between $\Delta^p(K)$ and $K$'s volume rigidity.
\subsection*{Basic properties of $\Delta^p(\cdot)$}
We start by briefly exploring some useful properties of the complex $\Delta^p(K)$ that appears in Theorem \ref{thm:main}.
Given a partially ordered set (poset) $(P,<)$ and an element $x\in P$ we denote by $P_{<,x}$ the prefix $\{y\in P~:~ y\le x\}$.
\begin{claim}
Let $(P, <)$ be a poset and $x \in P$, then there exists a linear extension $<_l$ of $<$ such that $P_{<_l,x} = P_{<,x}$.
\end{claim}
\begin{proof}
View the sets $A=P_{<,x}$ and $B=P\setminus P_{<,x}$ as posets with the partial order induced by $<$. Extend each of these posets linearly, and concatenate the extensions such that the elements in $A$ are smaller than those in $B$.
\end{proof}

We will mainly work with the partial order $<_p$ on the power set of $[n]$ and denote the set of its linear extensions by $\cL$. We usually denote an element in $\cL$ by $<_l$ and the corresponding shifted complex by $\Delta^l(K) := \Delta^{<_l}(K)$.

\begin{claim}\label{claim:cup_<_l}
$\Delta^p(K) = \bigcup_{<_l\in \cL} \Delta^l(K)$.
\end{claim}
\begin{proof}
On the one hand, if $\sigma\in\Delta^l(K)$
for some $<_l\in\cL$ then $q( f_\sigma)$ is not spanned by $B_{l,\sigma}:=\{q(f_\tau)~:~\tau <_l \sigma\}$, which contains the vector set $B_{p,\sigma}$. Therefore, by the definition of $\Delta^<$ in \eqref{eqn:Delta}, we find that $\Delta^p(K)\supseteq\Delta^l(K)$. On the other hand, for every $\sigma\in\Delta^p(K)$, there exists by the previous claim a linear extension $<_l\in\cL$ satisfying $B_{p,\sigma}=B_{l,\sigma}$ hence $\sigma \in \Delta^l(K)$.
\end{proof}

\begin{corollary}\label{cor:shifted}
For every simplicial complex $K$ there holds that $\Delta^p(K)$ is a shifted simplicial complex independent of the choice of the generic basis $f$. In addition, $\Delta^p(K) = K$ if $K$ is shifted. $\square$
\end{corollary}        
That $\Delta^p(K)$ is downwards closed follows exactly as in the proof for $\Delta^{\mathrm{lex}}(K)$. The rest of
Corollary \ref{cor:shifted} follows immediately from the above decomposition of $\Delta^p(K)$ and the fact that the basic properties of algebraic shifting in ~\cite{kalai90} -- being shifted, and independence the the generic $f$ chosen-- hold in every linear extensions of $<_p$, as remarked in~\cite[p.58]{kalai90}.

\subsection*{Volume rigidity and $\Delta^p(\cdot)$}

We are now ready to prove Theorem \ref{thm:main}. We denote $\sigma_0=\{1,3,...,d,n\}$ and observe that the prefix $B:=\{\tau\le_p\sigma_0~:~|\tau|=d\}$ consists of the subsets $[d]$ and  $[d]\setminus\{i\}\cup\{v\}$ for $
2\le i \le d$ and $d+1\le v\le n$. We define a linear transformation $\psi: \bigoplus_{i=2}^{d}\bigwedge^1 \R^n \rightarrow \bigwedge^d \R^n$ given by $$\psi(m_2,\dots ,m_d)  = \sum_{i=2}^d f_{[d]\setminus \{i\}}\wedge m_i.$$
\begin{lemma}\label{lem:kernel}
The image of $\psi$ is spanned by $\{f_\tau~:~\tau\in B\},$ and its kernel is $(d^2-d-1)$-dimensional.
\end{lemma}
\begin{proof}
The fact that $f_\tau\in\mathrm{im}(\psi)$ for every $\tau \in B$ can be shown directly. Indeed, $\psi(0,...,0,f_d)=f_{[d]}$ and by taking $m_i=f_v,~m_{i'}=0~\forall i'\ne i$ we have that $\psi(0,...,f_v,...,0)=f_{[d]\setminus\{i\}\cup\{v\}}$ for $2\le i \le d$ and $d+1\le v\le n$. To show that these $1+(n-d)(d-1)$ linearly independent vectors span the image of $\psi$, we will construct $d^2-d-1$ linearly independent vectors in $\ker\psi$ which actually completes the proof by the rank-nullity theorem since $(1+(n-d)(d-1)) + (d^2-d-1) = n(d-1).$

First, for every $2\le i\le d$ and $j\in[d]\setminus \{i\}$ consider the vector defined by setting $m_i=f_j$ and $m_{i'}=0$ for every $i'\ne i$. Then, 
\[
\psi(m_2,...,m_d) = f_{[d]\setminus\{i\}}\wedge f_j = 0
\]
since $j\in[d]\setminus \{i\}$. This amounts to $(d-1)^2$ vectors in $\ker\psi,$ and the remaining $d-2$ are given by vectors of the form $m_i=a_if_i,~2\le i \le d$,
where the scalars $a_2,...,a_d$ satisfy $\sum_{i=2}^{d} (-1)^ia_i = 0$. Indeed,
\[
\psi(m_2,...,m_d) = \sum_{i=2}^d a_i f_{[d]\setminus\{i\}}\wedge f_i = \sum_{i=2}^d a_i(-1)^{d-i}f_{[d]}=0\,.
\]
The linear independence of these $d^2-d-1$ vectors follows directly from the linear independence of $f_1,...,f_d$.
%
\end{proof}

\begin{proof}[Proof of Theorem \ref{thm:main}]

Identify the vertices of $K$ with the set $[n]$. 
W.l.o.g. assume that $f_{d-1}(K)\ge (d-1)n-(d^2-d-1)$, as otherwise $K$ is not volume-rigid and $\{1,3,4,...,d,n\} \notin \Delta^p(K)$.
In addition, suppose that the generic embedding $\p:V(K)\to\R^{d-1}$ is obtained from the vectors $f_2,...,f_d$ in the generic basis of $\R^n$ by taking $(f_i)_v=\p(v)_{i-1}$ for every $2\le i \le d$ and $v\in [n]$. 

Consider the $(i,v)$-unit vector $e_{i,v}=(m_2,...,m_d)$  in the domain of $\psi$, for  $2\le i \le d$ and $v\in [n]$, defined by $m_i=e_v$ and $m_{i'}=0,~\forall i'\ne i$. Then, 
\begin{align*}
\psi(e_{i,v})&=f_{[d]\setminus\{i\}}\wedge e_v\\&= (-1)^{d-i}f_1\wedge\dots\wedge f_{i-1}\wedge e_v \wedge f_{i+1}\wedge\dots\wedge f_d\\&=
(-1)^{d-i+d-1}f_2\wedge\dots\wedge f_{i-1}\wedge e_v \wedge f_{i+1}\wedge\dots\wedge f_d\wedge f_1.
\end{align*}
Let $\sigma =\{v_1,...,v_d\}\subset [n]$. Clearly, for the inner product on $\bigwedge \R^n$ with orthonormal basis $\{e_\sigma:\ \sigma \subset [n]\}$, we have 
$\langle e_\sigma, \psi(e_{i,v})\rangle = 0$  if $v\notin\sigma.$ Otherwise, by the identification of $\p$ with $f_2,...,f_d$ above and $f_1=\unitvec$, if $v=v_j$ then $\langle e_\sigma, \psi(e_{i,v})\rangle$ is equal to $(-1)^{i-1}$ times the determinant of the matrix that is obtained from $M_{\p,\sigma}$ by replacing  its $(i-1)$-th row with the $j$-th $d$-dimensional all-ones row vector. Consequently, 
$$
\langle e_\sigma, \psi(e_{i,v})\rangle = (-1)^{i-1}C_{i-1,j}(M_{\p,\sigma}).
$$
Thus, by letting $q: \bigwedge \R^n \longrightarrow \bigwedge K$ be the natural quotient map, and by choosing the basis $\{{e}_\sigma: \sigma \in K\}$ for $\bigwedge K$, we find that the $f_{d-1}(K)\times (d-1)n$ matrix representation of $q\circ\psi$ is equal --- up to multiplying some of its columns by $-1$ and reordering them --- to the transpose of the volume-rigidity matrix $\fV(K,\p)$. Therefore, $K$ is volume-rigid if and only if $\dim\ker(q\circ \psi)=d^2-d-1.$ In other words, $K$ is {\em not} volume-rigid if and only if there exists a non-zero $f\in\mathrm{im}(\psi)$ such that $q(f)=0$. By the characterization of $\psi$'s image in Lemma \ref{lem:kernel}, $f$ can be written as a non-trivial linear combination $f=\sum_{\tau\in B}\lambda_\tau f_\tau$. 

To conclude the proof we claim that $q(f)=0$ for some $f\in\SPAN\{f_\tau~:~\tau\in B\}$ if and only if $\sigma_0\notin \Delta^{p}(K)$. Indeed, on one direction, $q(\sum_{\tau\in B}\lambda_\tau f_\tau)=0$ implies that for some $\tau\in B$, $q(f_\tau)$ is a linear combination of its predecessors in $<_p$. By \eqref{eqn:Delta}, $\tau\notin\Delta^p(K)$ and since $\Delta^p(K)$ is shifted then $\sigma_0\notin\Delta^p(K)$. On the other hand, by manipulating the linear combination which asserts that $\sigma_0\notin\Delta^p(K)$, we obtain a non-zero vector $f=\sum_{\tau\in B}\lambda_\tau f_\tau$ satisfying $q(f)=0$.
\end{proof}

\section{Volume rigidity, local moves and homology}\label{sec:local}
We turn to study the effect of local combinatorial moves on volume rigidity. We start by proving a volume-rigidity analog of Whiteley's vertex splitting~\cite{Whiteley-split}, by which he showed that every triangulation of the $2$-sphere has a $3$-rigid $1$-skeleton. 

\begin{lemma}[Edge contraction]\label{lem:contraction}
  Let $K$ be a pure $(d-1)$-dimensional simplicial complex, $e =\{u,w\} \in K$ such that
   at least $(d-1)$ facets in $K$ contain $e$. Let $K'$ to be the simplicial complex obtained from $K$ by contracting the edge $e$, i.e. by identifying the vertex $u$ with $w$, and removing duplicates. If $K'$ is
  volume rigid then so is $K$.
\end{lemma}
\begin{proof}
Without loss of generality assume that $u<w$ are the first among the $n$ vertices of $K$, as the vertex labels do not effect volume-rigidity. We will construct an auxiliary $(d-1)n\times f_{d-1}(K)$ matrix $A$ such that 
\[
\rank{\fV(K,\p)}\ge \rank(A) = (d-1)n-(d^2-d-1).
\]
First, we replace the position of the vertex $w$, i.e. $\p(w)$, by the position of the vertex $u$, i.e. $\p(u)$. Formally we define a new (non-generic) placement of vertices $\p'$ that coincides with $\p$ on all vertices except $w$ on which we set it to equal to $\p(u)$. Clearly, since $\p$ is generic, there holds $\rank{\fV(K,\p)}\ge\rank{\fV(K,\p')}.$ To obtain $A$, we add the rows 
in $\fV(K,\p')$ corresponding to the vertex $u$ to the rows corresponding to the vertex $w$, an operation that  does not change the rank. 

We first claim that the submatrix of $A$ which corresponds to the columns of the facets $L$ containing $e=\{u,w\}$ is supported on the rows corresponding to $u$. Indeed,
if $e\subseteq \sigma=\{v_1,\ldots,v_d\}\in L$ then for $v_j\in \sigma$ such that $v_j\neq \{u,w\}$ we have that each entry $A_{v_j,i ; \sigma} =C_{i,j}(M_{\p',\sigma}) = 0$ because $M_{\p',\sigma}$ has two identical columns as $\p'(u) = \p'(w)$. On the other hand, because we added the rows corresponding to vertex $u$ to the rows corresponding to vertex $w$, we have that $$A_{w,i;\sigma} = C_{i,1}(M_{\p',\sigma}) + C_{i,2}(M_{\p',\sigma}) = 0.$$ This follows from  our assumption that $u$ and $w$ are the first two vertices hence their cofactors in $M_{\p',\sigma}$ have opposite signs, and they in fact cancel-out since and $\p'(u)=\p'(w)$. 

Second, we claim that the submatrix $A_{u,L}$ of $A$ corresponding to the $d-1$ rows of $u$ and the columns of $L$ has a full rank of $d-1$. We derive this claim by the assumption that $|L|\ge d-1$ and the fact that $(\p'(v)~:v\ne u)$ is generic. Indeed, consider a vector $x\in\R^{d-1}$ in the left kernel of $A_{u,L}$. A brief calculation yields that the orthogonality of $x$ and the column in $A_{u,L}$ corresponding to the facet $\sigma =\{u,w,v_2,...,v_d\}$ is equivalent to $x$ being in the span of $\p'(v_i)-\p'(w),~i=2,...,d$. By the assumption that $\p$ is generic, such $|L|\ge d-1$ constraints are only satisfied by $x=0$ hence 
\begin{equation}\label{eq:AuL}
\rank(A_{u,L})=d-1.
\end{equation}

Third, consider the complement submatrix $A':=A_{\{u\}^c,L^c}$ whose rows correspond to all the vertices except $u$, and columns to all the facets that are not in $L$. We observe that $A'$ contains as a submatrix the generic volume rigidty matrix $\fV(K',\p)$ --- where $\p$ is viewed here as a generic embedding of $V(K')=V(K)\setminus\{u\}$ into $\R^{d-1}$. Indeed, every facet $\sigma$ of $K'$ arises from a facet $\hat\sigma$ of $K$. 
\begin{itemize}
    \item If $u\notin \hat\sigma$ then $\sigma=\hat\sigma$ and the columns in $\fV(K',\p)$ and $A'$ corresponding to $\sigma$ are clearly equal.
    \item Otherwise, if $u\in \hat\sigma$ then $\sigma = \hat\sigma\cup\{w\}\setminus\{u\}$, and by the construction of $A$ --- in which $\p'(u)=\p'(w)$ and the rows of $u$ are added to the rows of $w$ --- we have that the column in $A'$ created from $\hat\sigma$ is equal to the column of $\sigma$ in $\fV(K',\p)$. 
\end{itemize}
 Note that $A'$ may contain some duplicate columns --- in case there are two facets that differ only in the vertices of $e$ --- but, regardless, our observation that $A'$ contains $\fV(K',\p)$ as a submatrix implies that 
\begin{equation}\label{eq:A'}
\rank(A')= (n-1)(d-1)-(d^2-d-1).    
\end{equation}

In conclusion, $A$ takes the form
\[
A=\bordermatrix
{
        ~ & L & L^c \cr
        ~~u & A_{u,L} & *\cr
        \{u\}^c & 0 & A' \cr
}, 
\]
and by combining \eqref{eq:AuL} and \eqref{eq:A'} we find that 
$$\rank(\fV(K,\p)\ge \rank(A)=\rank(A_{u,L})+\rank(A')=n(d-1)-(d^2-d-1)\,,
$$
as claimed.

\end{proof}

The next two lemmas are direct analogs of basic results in graph rigidity~\cite{AR2} asserting that gluing preserves volume-rigidity. We include their proofs for completeness.
\begin{lemma}
  Let $K$ be $(d-1)$-volume-rigid, $v\notin V(K)$ and $S\subseteq V(K)$ such that $|S|\geq d$, then $K \cup (v*\binom{S}{d-1})$ is $(d-1)$-volume-rigid.
\end{lemma}
\begin{proof}
  The vertex $v$ is in at least $d-1$ facets of $L = K \cup (v*\binom{S}{d-1})$. The volume rigidity matrix of $L$ is of the form
  \begin{equation*}
\bordermatrix
{
        ~ & K & S*v \cr
        V & \fV(K,\p) & *\cr
        v & 0 & N \cr
}, 
\end{equation*}
where the matrix $N$ has $d-1$ rows and at least $d-1$ columns. Because of general position $N$ has full rank, i.e. $rank(N) = d-1$. Because $K$ is $(d-1)$-volume-rigid we have that $rank(\fV(K,\p)) = n(d-1) - d(d-1) + 1$. Then, $rank(\fV(L,\p)) = rank(\fV(K,\p)) + rank(N) = (n+1)(d-1) - d(d-1) + 1$ and consequently $L$ is $(d-1)$-volume-rigid. 
\end{proof}

\begin{lemma}[Union of volume-rigid complexes]\label{lem:glue}
Let $K$ and $L$ be $(d-1)$-volume-rigid complexes such that $|V(K)\cap V(L)|\geq d$. Then $K\cup L$ is  $(d-1)$-volume-rigid.
\end{lemma}
\begin{proof}
Because $K$ and $L$ are $(d-1)$-volume-rigid we can assume that each of them has a complete $(d-1)$-skeleton on its respective vertex set. Then $K\cup L$ contains the vertex spanning subcomplex $Q$ obtained from $K$ by adding one vertex $v$ at a time, adding the facets $\binom{S}{d-1} * v$ at step $v$, where $S=V(L)\cap V(K)$. This subcomplex is $(d-1)$-volume-rigid at each step by application of the previous lemma. In particular, $Q$ is $(d-1)$-volume-rigid, hence so is  $K\cup L$.
\end{proof}

\subsection{Proof of Corollary \ref{cor:surfaces}}
\label{sub:cor_surface}
Barnette and Edelson~\cite{Barnette-Edelson-Orientable, Barnette-Edelson-NonOrientable} proved that every compact surface without boundary admits only finitely many \emph{irreducible} triangulations, namely, triangulations where every edge contraction would result in a simplicial complex not homeomorphic to the given surface.
Thus, by Lemma~\ref{lem:contraction}, in order to conclude that for a given surface $S$ every simplicial complex that triangulates it is volume-rigid, it is enough to verify if for the irreducible triangulations of $S$. 
Those are known for the surfaces indicated in Corollary~\ref{cor:surfaces}: one for the $2$-sphere (namely the boundary of a tetrahedron), 
two for the projective plane~\cite{Barnette:RP2}, 
21 for the torus~\cite{Lavrenchenko:Torus} 
and 29 for the Klein bottle~\cite{Lawrencenko-Negami:K, Sulanke:K}. 
Clearly the boundary of the tetrahedron is volume-rigid, and we verified by computer that the irreducible triangulations $K$ of the 
other surfaces mentioned above are volume-rigid -- for this it was enough to find some embedding $\p_K: V(K)\longrightarrow \R^2$ such that $\rank(\fV(K,\p_K)=2|V(K)|-5$. 

\begin{remark}
The fact that every triangulation $K$ of the $2$-sphere is volume-rigid follows also from combining the $3$-hyperconnectivity of its graph
with the Cohen-Macaulay property. Indeed, the first property says that $\{3,|V(K)|\}\in \Delta^{\mathrm{lex}}(K)$, and as $K$ is Cohen-Macaluay then $\Delta^{\mathrm{lex}}(K)$ is pure and hence $\{1,3,|V(K)|\}\in \Delta^{\mathrm{lex}}(K)$ as $\Delta^{\mathrm{lex}}(K)$ is shifted, which implies, by Claim~\ref{claim:cup_<_l}, that $\{1,3,|V(K)|\}\in \Delta^p(K)$, and we are done by Theorem~\ref{thm:main}.
\end{remark}

To prove the second part of Corollary \ref{cor:surfaces} we are left to show that removing one triangle from a triangulated $2$-sphere or torus preserves volume-rigidity, done next.
A pure simplicial complex is a \emph{minimal cycle} (over some coefficients commutative ring $\F$) if there exists an $\F$-linear combination of its facets whose boundary vanishes, and no proper nonempty subset of its facets has this property. For example, every triangulation of a compact connected surface (resp. and orientable) is a minimal cycle over $\Z_2$ (resp. $\Z$). 

\begin{lemma}\label{lem:minus_a_facet}
If $K$ is a $(d-1)$-dimensional 
volume rigid minimal cycle 
over $\Z$, then $K\setminus \sigma$ is volume rigid for
every
$\sigma \in K$.
\end{lemma}
We first give short proof for the special case $d=3$ and conclude the proof of Corollary \ref{cor:surfaces}. Afterwards, we give a more technical proof for the general case.

\begin{proof}[Proof of Lemma \ref{lem:minus_a_facet} ($d=3$).]
As $K$ is a minimal cycle over $\Z$, its $2$-dimensional homology with $\R$-coefficients is one dimensional, and for each facet $\sigma$ of $K$, for $K\setminus \sigma$ this homology vanishes.
By the translation of homology in terms of algebraic shifting, $\Delta^{\rm{lex}}(K)\ni \{2,3,4\} \notin  \Delta^{\rm{lex}}(K\setminus \sigma)$, and as shifting preserves containment we conclude \[\Delta^{\rm{lex}}(K\setminus \sigma)= \Delta^{\rm{lex}}(K)\setminus  \{\{2,3,4\}\}.\]

Note that in this dimension\footnote{For $d>3$, $\{\tau:\ \tau<_p \{1,3,4,\ldots,d,n\}\}$ is smaller than $\{\tau:\ \tau<_\mathrm{lex} \{1,3,4,\ldots,d,n\}\}$.} 
$\tau<_p \{1,3,n\}$ iff $\tau<_\mathrm{lex} \{1,3,n\}$, and thus: 
$\{1,3,n\} \in \Delta^p(K)$ (by Theorem~\ref{thm:main}), hence $\{1,3,n\} \in \Delta^{\mathrm{lex}}(K)$, and by the displayed equality above also
$\{1,3,n\} \in \Delta^{\mathrm{lex}}(K\setminus \sigma)$, so finally
$\{1,3,n\} \in \Delta^p(K\setminus \sigma)$, equivalently, 
$K\setminus \sigma$ is volume-rigid.
\end{proof}

\begin{proof}[Proof of Corollary~\ref{cor:surfaces}]
This is immediate from Lemma~\ref{lem:contraction}, Lemma~\ref{lem:minus_a_facet} for the case $d=3$, and the discussion in the beginning of \S\ref{sub:cor_surface}.
\end{proof}

We conclude this section with a more direct proof of Lemma \ref{lem:minus_a_facet} for all $d\ge 3$.

\begin{proof}[Proof of Lemma \ref{lem:minus_a_facet}]
For two subsets $\sigma$ and $\tau=\sigma\setminus\{v\}$ of $[n]$ that differ by one element, denote $\mathrm{sign}(\tau,\sigma):=(-1)^j$ if $v$ is the $j$-th element in $\sigma.$ We prove the following stronger statement. Let $z\in \Z^{f_{d-1}(K)}$ be a $(d-1)$-dimensional chain in $K$, and $\partial z \in \Z^{f_{d-2}(K)}$ be its boundary, i.e. $(\partial z)_{\tau} = \sum_{\tau \subset \sigma}{\mathrm{sign}(\tau,\sigma)}\cdot z_\sigma.$ Then, for every vertex $v$ and every $i\in[d-1]$
\begin{equation}\label{eq:bndry}
\left(\fV(K,\p)\cdot z\right)_{v,i} = (-1)^{d+i}\sum_{\rho}  {\mathrm{sign}(\rho,\rho\cup\{v\})}\cdot \det N_{\p,\rho,i}\cdot (\partial z)_{\rho\cup\{v\}}\,,    
\end{equation}
where (i) the summation is over the $(d-3)$-faces $\rho$ of $K$ that belong to the link of $v$ and, (ii) the $(d-2)\times(d-2)$ matrix $N_{\p,\rho,i}$ is obtained from the $(d-1)\times (d-2)$ matrix $(\p(v)~:~v\in\rho)$ by removing its $i$-th row. In particular, if $z$ is a generator of the $(d-1)$-homology of a minimal cycle $K$, then $z$ is also a non-trivial linear dependence between the columns of $\fV(K,\p)$. Therefore, removing a $(d-1)$-face from $K$ does not change the rank of the volume rigidity matrix, hence if $K$ is volume rigid then so is $K\setminus\{\sigma\}.$ To derive \eqref{eq:bndry}, note that
\begin{align}
\nonumber\left(\fV(K,\p)\cdot z\right)_{v,i}& = \sum_{v\in \sigma \in K} z_\sigma\cdot C_{i,j}M_{\p,\sigma} \\
\label{eq:det_expand}&=\sum_{\sigma} z_\sigma\cdot(-1)^{i+j}\sum_{j'\in[d]\setminus\{j\}} (-1)^{d-1+j'-\mathbf{1}_{j'>j}}\det N_{\p,\sigma\setminus\{v_j,v_j'\},i}
\end{align}
Indeed, suppose that $\sigma=\{v_1,...,v_d\}$ and $v=v_j$ and expand the $(i,j)$-th minor of $M_{\p,\sigma}$ by the last row (of ones). Denote $\tau := \sigma\setminus\{v_{j'}\}$ and $\rho:=\tau\setminus\{v_j\}$, and we easily observe that $\mathrm{sign}(\rho,\tau) =(-1)^{j-\mathbf{1}_{j'<j}}.$ Therefore, by changing the order of summation in \eqref{eq:det_expand} we find that
\[
\left(\fV(K,\p)\cdot z\right)_{v,i} = (-1)^{d+i}\sum_{\rho}\mathrm{sign}(\rho,\tau)\det N_{\p,\rho,i}\sum_{ \tau \subset \sigma} \mathrm{sign}(\tau,\sigma)\cdot z_\sigma,
\]
as claimed.
\end{proof}

\section{Volume rigidity and sparsity}\label{sec:sparse}
\begin{proof}[Proof of Corollary~\ref{cor:Sparsity-vs-VolRigidity}]
Let $d\ge 3$ and 
$K$ be the $(d-1)$-dimensional simplicial complex obtained from the graph $K_{3,3}$ by iterating the cone operation $d-2$
times. Then $K$ is $(d-1,d^2-d-1)$-sparse. (Indeed, $K_{3,3}$ is $(2,3)$-sparse, and if a pure $(k-1)$-dimensional simplicial complex is $(k-1,k^2-k-1)$-sparse then its cone is $(k,(k+1)^2-(k+1)-1)$-sparse.)
Thus, by completing it to a basis in the $(d-1,d^2-d-1)$-sparsity matroid,
we find a basis $K'$ containing $K$, so $K'$ is $(d-1,d^2-d-1)$-tight.

In order to show that $K'$ is not volume-rigid, by Theorem~\ref{thm:main} it is enough to show that \[\{1,2,\ldots,d-2,d+1,d+2\}\in \Delta^p(K'),\] as 
$\{1,2,\ldots,d-2,d+1,d+2\} \nleq_p \{1,3,4,\ldots,d,n\}$ and using tightness. 

The displayed equation above follows from basic properties of this shifting operator, proved in the same way as for exterior shifting w.r.t. the lex-order:
\begin{itemize}
    \item If $K$ is a subcomplex of $K'$ then $\Delta^p(K)\subseteq \Delta^p(K')$.
    \item Cone and $\Delta^p$ commute, namely, if $K=v*L$ for a simplicial complex $L$ then $\Delta^p(K)=1*(\Delta^p(L)+1)$. 
\end{itemize}
Here for a family $F$ of subsets of $[m]$, $F+1:=\{B+1:\ B\in F\}$, and $B+1:=\{i+1:\ i\in B\}$ (so $\emptyset+1=\emptyset$). 
To conclude the proof it is left to note that $\{3,4\}\in \Delta^{\rm{lex}}(K_{3,3})$ and hence, by Claim~\ref{claim:cup_<_l}, also $\{3,4\}\in \Delta^p(K_{3,3})$.
\end{proof}

\section{Concluding remarks}
\label{sec:open}
We end up with some related open problems.
An obvious one is to extend Corollary~\ref{cor:surfaces} to include all surface triangulations.
\begin{conjecture}
\label{conj:surfaces}
Every triangulation of a compact connected surface without boundary, minus a single triangle, is volume-rigid.
\end{conjecture}
The problem we face in applying Fogelsanger's decomposition~\cite{fogelsanger88} (see also~\cite[Sec.3.3]{CJT:newFog}) to volume rigidity of surfaces is that the pieces in the decomposition include triangle faces not existing in the original triangulation, and thus the gluing lemmas we could prove, e.g. Lemma~\ref{lem:glue}, are not strong enough to settle Conjecture~\ref{conj:surfaces}.

It is known that for every triangulation $K$ of the $2$-sphere on $n$ vertices, minus a single triangle, its exterior shifting $\Delta^{\rm{lex}}(K)=\Delta^p(K)$ consists exactly of the triangle $13n$ and all the triangles that are smaller than it in the lex-order, and their subsets.  This is a sufficient condition for volume rigidity by Theorem \ref{thm:main}. The following conjecture deals with  a higher-dimensional counterpart of this fact.

\begin{conjecture} For every $d\ge 3$, every triangulation $K$ of the $(d-1)$-sphere minus a single $(d-1)$-simplex is volume rigid.
\end{conjecture}

It is also natural to ask whether the stronger property of $\{1,3,4,\ldots,d,n\} \in \Delta^{\rm{lex}}(K)$ holds true. This is known, and tight, for stacked spheres~\cite{Murai:stacked} (also \cite[Example 2.1.8]{Nevo:thesis}).
Let us remark that the conclusion  $\{1,3,4,\ldots,d,n\} \in \Delta^s(K)$ for Kalai's symmetric shifting operator $\Delta^s(\cdot)$ is equivalent to the hard-Lefschetz isomorphism from degree $1$ to degree $d-1$ in a generic Artinian reduction of the Stanley-Reisner ring of $K$ over the field of reals; the later isomorphism was proved recently by Adiprasito~\cite{Adi:beyond}. 


Back to general complexes, 
\begin{problem}\label{prob:CombiChar}
For every dimension, find a combinatorial characterization of the corresponding volume-rigidity matroid.
\end{problem}
The combinatorial characterization problem is important for the $d$-rigidity matroid (and is open for $d\geq 3$). The $d$-rigidity of a graph $G$ on $n$ vertices is equivalent to 
$\{d,n\}\in \Delta^s(G)$. In view of this fact, we ask:
\begin{problem}
Define a version of symmetric shifting $\Delta^{sp}(\cdot)$ and find a matroid on $\binom{[n]}{d}$ such that its bases $K$ are exactly those satisfying $\Delta^{sp}(K)=\{\tau: \tau\le_p \{1,3,4,\ldots,d,n\}\}$. 
\end{problem}

An additional direction to explore is the volume rigidity of a $(d-1)$-dimensional simplicial complex $K$ in $\R^{d'}$  for $d' \ge d-1$. That is, let $\p:V(K)\to\R^{d'}$ be generic, and ask whether there is a non-trivial motion of the vertices that preserves all the volumes of $K$'s $(d-1)$-dimensional simplices in $\R^{d'}$. The case $d=2$ corresponds to  the standard framework rigidity in $\mathbb R^{d'}$, and the case $d'=d-1$ is the volume rigidity notion we study in this paper. Several natural questions on the remaining cases $2<d\le d'$ arise: what are the trivial motions in this setting? Is there a characterization of $(d-1)$-volume rigidity in $\R^{d'}$ in terms of algebraic shifting?

\bibliographystyle{alpha}
\bibliography{main}

\end{document}